\documentclass{amsart}
\usepackage{amsmath,amsthm}
\usepackage{amsfonts,amssymb}
\usepackage{enumerate}
\usepackage{accents,color}
\usepackage{graphicx}

\newcommand{\lvt}{\left|\kern-1.35pt\left|\kern-1.3pt\left|}
\newcommand{\rvt}{\right|\kern-1.3pt\right|\kern-1.35pt\right|}

\newtheorem{thm}{Theorem}

\newtheorem{lem}{Lemma}

\newtheorem{prob}{Problem}
\newtheorem{THEO}{Theorem}

\newtheorem{LEM}{Lemma}

\theoremstyle{remark}

\makeatletter
\newcommand{\bddots}{%
  \mathinner{\mkern1mu\raise\p@\vbox{\kern7\p@\hbox{.}}\mkern2mu
    \raise4\p@\hbox{.}\mkern2mu\raise7\p@\hbox{.}\mkern1mu}}
\makeatother
 \def \to {\rightarrow}

\def \N {\mathbb{N}}

\def \R {\mathbb{R}}
\def \C {\mathbb{C}}

\graphicspath{{./}}

\begin{document}

\title[Nyman-Beurling and B\'aez-Duarte criteria]
{An extremal problem related to generalizations  of the Nyman-Beurling and B\'aez-Duarte criteria}

\author{Dimitar K. Dimitrov}
\address{Departamento de Matem\'atica Aplicada\\
 IBILCE, Universidade Estadual Paulista\\
 15054-000 Sa\~{o} Jos\'e do Rio Preto, SP, Brazil.}
 \email{dimitrov@ibilce.unesp.br}
\author{Willian D. Oliveira}
\address{Departamento de Matem\'atica Aplicada\\
 IBILCE, Universidade Estadual Paulista\\
 15054-000 Sa\~{o} Jos\'e do Rio Preto, SP, Brazil.}
 \email{wdoliveira@ibilce.unesp.br}
 
 \thanks{Research supported by by the Brazilian foundations CNPq under
Grant 307183/2013-0 and FAPESP under Grants 2016/09906-0 and 2013/14881-9.}


\begin{abstract} We establish generalizations of the Nyman-Beurling and B\'aez-Duarte criteria concerning lack of zeros of Dirichlet $L$-functions 
in the semi-plane $\Re(s) >1/p$ for $p\in (1,2]$.  We pose and solve a natural extremal problem for Dirichlet polynomials which take values one 
at the zeros of the corresponding $L$-function on the vertical line $\Re(s)=1/p$.  
\end{abstract}

\maketitle

\section{Introduction and statement of results}
\setcounter{equation}{0}

Let $C$ be the space of functions $h:(0,1) \mapsto \C$ of the form
\begin{equation}
\label{C}
h(x)=\sum_{k=1}^{n}b_k \left\{\frac{1}{\theta_kx}\right\}, \ \ \ \ \ \theta_k \geq1, \ \ \ \ \ n\in \mathbb{N},
\end{equation}
where $\{x\}=x-[x]$ denotes the fractional part of $x$ and the constants $b_k\in \C$ obey the restriction $\sum b_k/\theta_k=0$. 
Let $C^p$ be the closure of $C$ in $L^p(0,1)$. The following classical result is due to Beurling \cite{Beu55}:
\begin{THEO}
\label{NB}
The Riemann zeta function $\zeta(s)$ does not vanish in the semi-plane  $\Re(s)>1/p$ if and only if  $C^p=L^p (0, 1)$.
\end{THEO}
Since the result was proved first by Nyman \cite{Nym} for $p=2$ in 1950 and generalised by Beurling \cite{Beu55} for $p>1$ in 1955 
it is nowadays commonly known as the Nyman-Beurling criterion. 
Bercovici and Foias \cite{BerFoi} establish the case $p=1$ in 1984 while it obviously does not hold for $p>2$. Theorem \ref{NB} has attracted the attention because of its importance to the study of distribution of zeros 
of the Riemann zeta function. We refer to the more recent contributions \cite{B-D93, B-D99, B-D03, B-D05, BBLS, Bag, BS, BalRot, BetConFar, Bur, DFMR13, DFMR13II, Rot05, Rot07, Rot09, Nik, Vas}. For $p=2$ 
the above theorem provides a criterion for the Riemann hypothesis (RH). 

In a sequence of papers B\'aez-Duarte \cite{B-D93, B-D99, B-D03, B-D05}, also in collaboration with Balazard, Landreau and  Saias \cite{BBLS}, obtained 
various improvements of the Nyman-Beurling criterion. In particular, in \cite{B-D03} B\'aez-Duarte 
showed that in the most important case $p=2$ the conditions $\theta_k\geq1$ in (\ref{C}) can be substituted by $\theta_k\in \N$ and the restriction $\sum b_k/\theta_k=0$ can be removed. B\'aez-Duarte's 
contribution implies the following beautiful criterion for the Riemann hypothesis in terms of approximation of the characteristic function $\textbf{1}_{(0,1)}$ of the interval $(0,1)$: 
\begin{THEO}
\label{B}
The RH holds if and only if 
$\lim_{n \to \infty}d_n=0$, where
$$
d_n^2=\inf_{{b_1,\ldots,b_n\in \C}\atop}\int_0^\infty \left|\textbf{1}_{(0,1)}-\sum_{k=1}^n b_k \left\{\frac{1}{kx}\right\}\right|^2 dx.
$$
\end{THEO} 

The beauty of the latter statement is that the above extremal problem is nothing but a problem about the best approximation of $\textbf{1}_{(0,1)}$ in a 
Hilbert space in terms of elements from a finite dimensional subspace and the solution of every such a problem is given by the projection. Indeed, $d_n$ 
is the distance in $L^2(0,\infty)$ from $\textbf{1}_{(0,1)}$ to the $n$-dimensional space $\mathrm{span} \{\rho_k(x):k=1,\ldots,n\}$, where $\rho_k(x)=\{1/kx\}$. It is well known that
$$
d_n^2 =  \frac{\mathrm{det}\ G(\rho_1,\ldots ,\rho_n, \textbf{1}_{(0,1)})}{\mathrm{det}\ G(\rho_1,\ldots ,\rho_n)},
$$
where $G(\rho_1,\ldots ,\rho_n, \textbf{1}_{(0,1)})$ and $G(\rho_1,\ldots ,\rho_n)$ are the Gram matrices of the corresponding functions 
and the inner product is defined by 
$$
(g,h) =  \int _0^\infty g(x)\, \overline{h(x)}\, dt.
$$
The B\'aez-Duarte useful version, including the one about $L$-functions in the Selberg class, due to de Roton \cite{Rot09},
have been proved only for $p=2$. We fill this gap, establishing a generalisation for Dirichlet $L$-functions, for every $p\in (1,2]$. In order to formulate our results, let $L(s,\chi)$ be a Dirichlet $L$-function with a character modulo $q$ and $p\in(1,2]$. Define the function
$$
\kappa(x)=\beta\, x^\alpha-\sum_{k\leq x}\chi(k)\, k^{1/2-1/p},
$$
where  $\alpha=3/2-1/p$ and $\beta= \varphi(q)/(\alpha q)$ if  $\chi$ is principal and $\alpha=\beta=0$ if  $\chi$ is non-principal character. We prove:
\begin{thm}
\label{GBNew}  
Let $p\in (1,2]$ and $L(s,\chi)$ be a Dirichlet $L$-function. Then $L(s,\chi)$ does not vanish for $\Re s>1/p$ if and only if 
$\lim_{n \to \infty}d_n(L,p)=0$, where
$$
d_n^2(L,p)=\inf_{{b_1,\ldots,b_n\in \C}\atop}\int_0^\infty \left|\textbf{1}_{(0,1)}-\sum_{k=1}^n b_k\, \kappa \left(\frac{1}{kx}\right) \right|^2 dx.
$$
 \end{thm}
 Observe that $d_n^2(L,p)$ is well defined. Indeed, for any $k \in \N$ the function $\kappa \left(\frac{1}{kx}\right)$ belongs to $L^2(0, \infty)$ because it vanishes for $x>1$ and 
 we shall prove that the integrand is bounded in $(0,1)$.
 
The relevance of the last results is that it permits the use of the classical tools from the theory of Hilbert spaces in the investigation of the zeros of Dirichlet  $L$-function on semi-planes $\Re(s)>1/p$ for $p\in (1,2]$. 
For instance, our generalization of Theorem \ref{B} allows us to write $d_n^2(L,p)$ as a quotient of determinants. More specifically, 
$$
d_n^2(L,p) =  \frac{\mathrm{det}\ G(\lambda_1,\ldots ,\lambda_n, \textbf{1}_{(0,1)})}{\mathrm{det}\ G(\lambda_1,\ldots ,\lambda_n)},
$$
where $\lambda_k(x)=\kappa(1/kx)$.

With the aid of the Mellin transform we prove that Theorem \ref{GBNew} is equivalent to the following one:
\begin{thm}
\label{GBLp}
Let $p\in (1,2]$ and $L(s,\chi)$ be a Dirichlet $L$-function. Then $L(s,\chi)$ does not vanish for $\Re s>1/p$ if and only if 
$\lim_{n \to \infty}d_n(L,p)=0$, where
$$
d_n^2(L,p)=\inf_{{A_n\in \mathcal{D}_n}\atop} \frac{1}{2\pi} \int_{\Re(s)=1/p}\left| \frac{1-L(s,\chi)A_n(s)}{s} \right|^2 |ds|
$$
and $\mathcal{D}_n$ is the space of ordinary Dirichlet polynomials of the form $\sum_{k=1}^n b_k k^{-s}$.
\end{thm} 

As we have already mentioned, Theorem 2 was established for $p=2$ for the more general class of Selberg $L$-functions by de Roton \cite{Rot09}. Since our principal aim is to extend the B\'aez-Duarte criterion 
to semi-planes $\Re(s)>1/p$ free of zeros of $L$-functions, we restrict ourselves to Dirichlet $L$-series. As it will become clear in the course of the proofs our choice is due the fact that we are able to deal with the reciprocal 
$1/L(s,\chi)$ when $L(s,\chi)$ is a Dirichlet $L$-series.

Since the quantity $d_n^2(L,p)$ can be defined equivalently by 
\begin{equation}
\label{dnz}
d_n^2(L,p) =\inf_{{A_n\in \mathcal{D}_n}\atop}  \frac{1}{2\pi} \int_{\mathbb{R}} \left| 1-L \left(\frac{1}{p}+it, \chi\right)\, A_n\left(\frac{1}{p}+it\right) \right|^2 \frac{dt}{1/p^2+t^2},
\end{equation}
then the fundamental question arises:
\begin{prob}
\label{P1}
 For any fixed $n\in \mathbb{N}$ and $p\in(1,2]$, determine the best approximation of $\textbf{1}_{(-\infty,\infty)}$ by 
products  of the form $L(1/p+it) A_n(1/p+it) $ in $L^2(\mathbb{R}, \omega)$, where $\omega(t)=1/(1/p^2+t^2)$ is the weight function.
\end{prob} 
Since, for any choice of $A_n$, the functions $1-L(s,\chi) A_n(s)$ which appear in (\ref{dnz}) are featured by the property that they take value one at the zeros of $L(s,\chi)$, we pose and solve the following natural extremal problem for Dirichlet polynomials:
\begin{prob}
\label{P2} Let $m\in \mathbb{N}$, $t_1, \ldots,t_m$ be $m$ distinct real numbers  and $\mathcal{D}_{n,m}^p$, $m\ll n$, be the space of Dirichlet polynomials $B_{n,p}$ of degree $n$ 
which obey the $m$ interpolation conditions $B_{n,p}(1/p+it_j) =1$, $j=1,\ldots, m$. Determine 
$$
d_{n,m,p}^2 =\inf_{{B_{n,p}\in \mathcal{D}_{n,m}^p}\atop}  \frac{1}{2\pi} \int_{\mathbb{R}} \left| B_{n,p}\left(\frac{1}{p}+it\right) \right|^2 \frac{dt}{1/p^2+t^2}.
$$
\end{prob}
We provide the following solution to Problem \ref{P2}:
\begin{thm}
\label{DnL} 
For every $m\in \mathbb{N}$ and for any distinct real numbers $t_1,\ldots, t_m$, there exists  $n(m)\in \mathbb{N}$, such that, for every $n > n(m)$, there is a unique $\tilde{B}_{n,p} \in \mathcal{D}_{n,m}^p$ for which the infimum 
$$
d_{n,m,p}^2=\inf_{{B_{n,p} \in \mathcal{D}_{n,m}^p}\atop} \frac{1}{2\pi} \int_{\Re(s)=1/p}\left| \frac{B_{n,p}(s)}{s} \right|^2 |ds|
$$
is attained. Moreover,
$$
d_{n,m,p}^2 \sim \frac{1}{\log n} \sum_{j=1}^m  \frac{1}{1/p^2+t_j^2},\ \ \mathrm{as}\ \ n\rightarrow \infty.
$$
\end{thm}

Substantial efforts have been put to guess which should be the sequence of Dirichlet polynomials $A_n$ for which 
eventually $d_n(\zeta,2)$ would converge to zero. Balazard and de Roton \cite{BalRot} proved that, under the RH, $d_n^2(\zeta,2) \ll (\log \log n)^{5/2+\varepsilon} (\log n)^{-1/2}$.   Despite that one of the natural candidates are the partial sums of $1/\zeta(s)$, that is 
$\zeta^{-1}_n(s) := \sum_{k=1}^n \mu(k) k^{-s}$, B\'aez-Duarte \cite{B-D99} proved that for this choice the corresponding quantity $d_n(\zeta,2)$ does not converge to zero. The ``molified'' partial sums, defined by 
$$
V_n(s) = \sum_{k=1}^n \left(1-\frac{\log k}{\log n} \right) \frac{\mu(k)} {k^{s}},
$$
seem to be a better choice. Recently Bettin, Conrey and Farmer \cite{BetConFar}  proved that, if the RH is true and the additional 
requirement that 
$\sum_{|\Im \rho|\leq T} 1/|\zeta^\prime(\rho)|^2  \ll T^{3/2-\delta}$  holds for some $\delta>0$, where as usual $\rho$ denote the nontrivial zeros of $\zeta$,  then 
$$
 \frac{1}{2\pi} \int_{\mathbb{R}} \left| 1-\zeta \left(\frac{1}{2}+it\right)\, V_n\left(\frac{1}{2}+it\right) \right|^2 \frac{dt}{1/4+t^2} \sim \frac{1}{\log n}\sum_{\Re (\rho) =1/2} \frac{1}{|\rho|^2},\ \ \mathrm{as}\ \ n\rightarrow \infty.
$$
Is is worth mentioning that  Burnol \cite{Bur}
had generalized a previous result of B\'aez-Duarte, Balazard, Landreau and Saias \cite{BBLS}  proving unconditionally that 
$$
\liminf_{n\rightarrow \infty}\ d_n^2(\zeta,2)\, \log n\ \geq \sum_{\Re(\rho)=1/2} \frac{m(\rho)^2}{|\rho|^2},
$$ 
where $m(\rho)$ stands for the multiplicity of $\rho$.  De Roton \cite{Rot06} extended the latter lower bound for $L$-functions in the Selberg class.
 
If we consider $t_1,\dots, t_m$ as the imaginary parts of $m$ distinct zeros $\rho_1,\ldots, \rho_m$ of the $\zeta$ function on the critical line, then the result of Theorem \ref{DnL} becomes
$$
d_{n,m,2}^2 \sim \frac{1}{\log n} \sum_{j=1}^m  \frac{1}{|\rho_j|^2},\ \ \mathrm{as}\ \ n\rightarrow \infty.
$$
Since this result is obviously not conditional, the similarity with the conditional result of Bettin, Conrey and Farmer is amazing. 
This raises the question about how close the spaces and the corresponding approximating functions in the extremal Problems 
1 and 2 are, especially if we let both $n$ and $m$ to go to infinity in Problem 2, eventually in a peculiar way. It is worth mentioning that
naive numerical experiments show that $1- \zeta V_n$ and $\tilde{B}_{n,p}$ match very well when the variable $t$ is close to the origin, 
even for relatively small values of $m$ and $n$.  

Suppose that the $L$-function $L(s,\chi)$, defined above, possesses $m$ distinct zeros $1/p+it_1,\ldots$, $1/p+it_m$ on the vertical line $\Re(s)=1/p$. Then Theorem \ref{DnL} says that   
$$
d_{n,m,p}^2 \sim \frac{1}{\log n} \sum_{j=1}^m  \frac{1}{|1/p+t_j|^2},\ \ \mathrm{as}\ \ n\rightarrow \infty.
$$

\section{Preliminary results}
\setcounter{equation}{0}

 In this section we provide definitions and various classical results that we shall need in the proofs of the main results in order to make 
the exposition relatively self-contained. Some technical results are established too.
\subsection{Definitions and general considerations}

 We denote by $\mathcal{D}$ the space of Dirichlet series $\eta(s)=\sum_{k=1}^{\infty}a_kk^{-s}$ with the following property:  $a_1\neq 0$ and there exist real constants 
$\alpha=\alpha_\eta$ and $\beta=\beta_\eta$, with $0 \leq \alpha \leq 1$, such that the function $\kappa_\eta:\R_+\mapsto \C$, defined by
\begin{equation}
\label{kapeta}
\kappa_\eta(x)=\beta\, x^\alpha-\sum_{k\leq x}a_k,
\end{equation}
is bounded. By convention,  we set $\beta=0$ whenever $\alpha=0$. With every $\eta \in \mathcal{D}$ we associate the space 
\begin{equation}
\label{Ceta}
C_\eta := \left\{ h:(0,1) \mapsto \C\,:\, h(x)=\sum_{k=1}^{n} b_k\, \kappa_\eta \left(\frac{1}{\theta_kx}\right), \ b_k\in \C, \ \theta_k \geq1, \ \beta\, \sum_{k=1}^{n} \frac{b_k}{\theta_k^{\alpha}}=0 \right\}.
\end{equation}
Observe that the functions in $C_\eta$ are well defined and  vanish identically for $x>1$. 
The subspace $C_{\eta,\N}\subset C_\eta$ is obtained from $C_\eta$ when the restrictions $\theta_k\geq 1$ are substituted by $\theta_k \in \N$. 
We denote by $C_{\eta}^p$ and $C_{\eta,\N}^p$  the closures of $C_\eta$ and $\C_{\eta,\N}$ in $L^p(0,1)$.

 Let $\eta \in \mathcal{D}$. Then for every complex function $f \in C^1(0,n]$, with the aid of Abel's  identity (see \cite[Theorem 4.2]{Apo}), we can write 
$$
\sum_{k\leq n} a_kf(k)=-\kappa_\eta(n)f(n)+\beta f(1)+\int_1^n \kappa_\eta(y)f'(y)dx+\beta\, \alpha \int_1^n y^{\alpha-1}f(y)dy.
$$
Choosing $f(y)=y^{-s}$ we obtain
$$
\sum_{k\leq n} \frac{a_k}{k^s}=-\frac{\kappa_\eta(n)}{n^s}+\beta-s\int_1^n \frac{\kappa_\eta(y)}{y^{s+1}}dy+\frac{\beta \alpha\, n^{\alpha-s}}{\alpha-s} -\frac{\beta \alpha}{\alpha-s}\ \  \ \mathrm{for}\ \ \Re(s)>1.
$$
Letting $n \to \infty$ and using the fact that $\kappa_\eta$ is bounded, we obtain
$$
\frac{\eta(s)}{s}=\frac{\beta}{s-\alpha}-\int_1^\infty \frac{\kappa_\eta(y)}{y^{s+1}}\, dy,  \ \ \ \Re(s)>0.
$$
If $\theta\geq1$ and $\Re(s)>0$, the change of variables $y=1/(\theta x)$ and the explicit expression for $\kappa_\eta$ yield
\begin{eqnarray*}
\frac{\eta(s)\, \theta^{-s}}{s}& = & \frac{\beta\, \theta^{-s}}{s-\alpha}-\int_0^{1/\theta} \kappa_\eta(1/\theta x)\, x^{s-1}\, dx\\
\ & = & \frac{\beta\, \theta^{-s}}{s-\alpha}-\int_0^{1} \kappa_\eta(1/\theta x)\, x^{s-1}\, dx + \int_{1/\theta}^{1} \kappa_\eta(1/\theta x)\, x^{s-1}\, dx\\
\ & = & \frac{\beta\, \theta^{-\alpha}}{s-\alpha}-\int_0^1 \kappa_\eta(1/\theta x)\, x^{s-1}\, dx.\\
\end{eqnarray*}
Therefore, for every $h(x)=\sum_{k=1}^{n} b_k\, \kappa_\eta (1/\theta_kx)$ in $C_\eta$ we obtain
\begin{equation}
\label{Int1f}
\int_0^1 h(x)x^{s-1}dx=-\frac{\eta(s)\sum_{k=1}^{n}b_k\theta_k^{-s}}{s}, \ \  \Re(s)>0.
\end{equation}
Since the Mellin transform is defined by 
$$
\mathcal{M}\left[f(x);s\right]=\mathcal{M}f(s)=\int_0^\infty f(x)x^{s-1}dx,
$$
then (\ref{Int1f}) shows that the Mellin transform of any function $h$ in $C_\eta$
is given by 
\begin{equation}
\label{Mellin}
\mathcal{M}\left[h(x);s\right]=-\frac{\eta(s)\sum_{k=1}^{n}b_k\theta_k^{-s}}{s}, \ \  \Re(s)>0.
\end{equation}

Let $L^2(\Re(s)=1/2)$ be the space of complex functions $f$ such that $g(t)=f(1/2+it)$ and  $g\in L^2(-\infty,\infty)$. The Mellin transform $\mathcal{M}:  L^2 (0,\infty) \mapsto L^2(\Re(s)=1/2)$ is defined in the following way. 
Let $\mathcal{B}\subset L^2 (0,\infty)$ be the space of functions $f$, such that the corresponding function $g(x)=f(x)x^{-1/2}$ obeys $g\in L^1(0,\infty)$. Then the Mellin transform of $f\in \mathcal{B}$ is defined by 
$$
\mathcal{M}\left[f(x);s\right]=\mathcal{M}f(s)=\int_0^\infty f(x)x^{s-1}dx.
$$
Observe that the Mellin and the Fourier transform are related by
$$
\mathcal{M}f(1/2+it)=\mathcal{F}\left[ f(e^{-u})e^{-u/2};t\right].
$$
This observation and the fact that $\mathcal{B}$ is dense in $L^2(0,\infty)$ allows us to apply the Plancherel theorem to extend $1/\sqrt{2 \pi} \mathcal{M}$ to an isometry between $L^2(0,\infty)$ and $L^2(\Re(s)=1/2)$.

\subsection{Lemmas and Theorems}

The following form of the Phragm\'en-Lindel\"of principle appear as Theorem 5.53 in  \cite{IwaKow}:
\begin{LEM} 
\label{PLP}
Let $f$ be a function holomorphic on an open neighborhood of a strip $a\leq \sigma \leq b$, for some real numbers $a<b$, such that
$$|f(s)|=O\left( \exp(|s|^C)\right)$$
for some $C\geq 0$ on $a\leq \sigma \leq b$. Assume that
\begin{eqnarray*}
|f(a+it)| & \leq & M_a(1+|t|)^\alpha \\
|f(b+it)| & \leq & M_b(1+|t|)^\beta
\end{eqnarray*}
for $t\in \R$. Then
$$
|f(\sigma+it|\leq M_a^ {l(\sigma)}M_b^{1-l(\sigma)}(1+|t|)^{\alpha l(\sigma)+\beta(1-l(\sigma))}
$$
for all $s$ in the strip, where $l$ is the linear function such that $l(a)=1$ and $l(b)=0$.
\end{LEM}

The following is Lemma 3.12 in  \cite{Tit}:

\begin{LEM}
\label{LemTit}
 Let $f(s)=\sum a_k k^{-s}$ be a Dirichlet series, convergent for $\Re(s)=\sigma >1$, with $a_k=O(\psi(k))$, where $\psi(n)$ is a nondecreasing function and
$$
\sum \frac{|a_k|}{k^{\sigma}}=O\left( \frac{1}{(\sigma-1)^\alpha} \right), \ \ \ \sigma \to 1.
$$ 
Moreover, if $c>0$, $\sigma+c>1$, $x$ is a non-integer and  $N$ is the integer closest to $x$, then 
\begin{eqnarray*}
\sum_{k<x} \frac{a_k}{n^s} & = & \frac{1}{2 \pi i} \int_{c-iT}^{c+iT}f(s+w)\frac{x^w}{w}dw+O\left( \frac{x^c}{T(\sigma+c-1)^\alpha}\right)\\
& & +O\left( \frac {\psi(2x)x^{1-\sigma} \log x}{T}\right) + O\left( \frac {\psi(N)x^{1-\sigma}}{T|x-N|}\right)
\end{eqnarray*}
\end{LEM}

Another result we shall need is the following summary of Theorems 5.6 and 5.23 in \cite{IwaKow}:

\begin{THEO} 
\label{ThIw}
Let $L(s,\chi)$ be a Dirichlet $L$-function. 
Then:
\begin{enumerate}
\item[(i)] the estimate 
$$
| L(s,\chi) | < O(|s|^{1/4})
$$
holds for $\Re(s)= 1/2$;
\item[(ii)] If $L(s,\chi)$ is a Dirichlet $L$-function with a primitive character modulo $q$, then the corresponding $\xi$-function
\begin{equation}
\label{xiL}
\xi_{\chi}(s)=(s(s-1))^{r(\chi)}q^{s/2}\pi^{s/2}\Gamma\left( \frac{s+\alpha(\chi)}{2}\right) L(s,\chi),
\end{equation}
where $r(\chi)$ is the order of the pole at $s=1$ and either $\alpha(\chi)=0$ if $\chi(-1)=1$ or $\alpha(\chi)=1$ when $\chi(-1)=-1$, 
is an entire function of order one and can be factorised in the form
\begin{equation}
\label{Had_fac}
\xi_{\chi}(s)=e^{a+bs}\prod_{\rho} \left(1-\frac{s}{\rho}\right)e^{s/\rho},
\end{equation}
where $\rho$ runs over the nontrivial zeros of $L(s,\chi)$. Moreover, if $L(s,\chi)$  is a Dirichlet $L$-function with any character modulo $q$, the Hadamard factorization {\rm (\ref{Had_fac})}  still holds because 
of \cite[Theorem 12.9]{Apo}. 
\end{enumerate}
\end{THEO}

The following is an analog of a celebrated result of Littlewood \cite{Litt} (see also \cite[Theorem 14.2]{Tit} and \cite[Theorem 1.12]{Ivic}) about conditional estimates of the zeta function along vertical lines on the semi-plane $\Re(s)> 1/2$. 
The generalisation below concerns estimates of a Dirichlet $L$-series and its reciprocal along vertical lines on the semi-plane free of zeros of $L$.
\begin{thm} 
\label{Littl_Est}
Let $\delta, \epsilon >0$, $p\in (1,2]$, and $L(s,\chi)$ be a Dirichlet  $L$-function without zeros in the semi-plane $\sigma>1/p$. 
Then there is a positive  $t_0=t_0(p,\delta,\epsilon)$ such that 
$$
|L(s,\chi)|^{\pm1}\leq |t|^\epsilon
$$
for  $\sigma > 1/p+\delta$ and $|t|\geq t_0$.
\end{thm} 

\begin{proof}  Let $\sigma>1/p$. 
Applying the Borel-Carath\'eodory theorem for $\log L(s,\chi)$ which is holomorphic for $\sigma>1/p$, with a possible exception at $s=1$,
and  the concentric circumferences with centre at $2+i t$ and radii
$$
2-\frac{1}{p}-\frac{1}{2\log \log t} \ \ \ \mathrm{and} \ \ \ 2-\frac{1}{p}-\frac{1}{\log \log t}, \ \ \ t>t_1,
$$
where $t_1$ is chosen in such a way that the possible pole at $s=1$ of $L(s,\chi)$ is outside the circumferences and simultaneously 
the inequality $\Re (\log L(s,\chi))=\log |L(s,\chi)|< \log t$ holds. The latter holds because of Theorem \ref{ThIw}, via an application of the Phragm\'en-Lindel\"of convexity principle.  
Then in the smaller circumference 
\begin{eqnarray}
|\log L(s,\chi)|  &\leq  &((8-\frac{4}{p}) \log \log t -4))\log t\\
& &  + ((8-\frac{4}{p}) \log \log t -3) |\log L(2+it,\chi)| \nonumber\\
&<  &8(\log \log t)(\log t+|\log L(2+it,\chi)|). \label{M3}
\end{eqnarray}
Since $\log L(2+it,\chi)= O(1)$, then $\log L(2+it,\chi) \leq \log t$ for $t>t_2$,  and there exists a constant $A>2$ such that 
$$
|\log L(s,\chi)|\leq A(\log \log t)\log t, \ \ \ t>\max\{t_1,t_2\}.
$$
Let $s$ be such that $1/p+1/(\log \log t)\leq \sigma \leq 1$. Let us apply the Hadamard three circles theorem to $C_1$, $C_2$ e $C_3$ with centre $\log \log t +it$ which pass 
through the points $1+(1/\log \log t)+it$, $\sigma +it$ and $1/p+(1/\log \log t)+it$. Then the radii are
$$
r_1=\log \log t-1-(1/\log \log t),\  r_2=\log \log t-\sigma\  \mathrm{and}\ \ r_3=\log \log t-1/p-(1/\log \log t),
$$
respectively. 
Denote by $M_1$, $M_2$ and $M_3$ the maxima of $|\log L(s,\chi)| $ on $C_1$, $C_2$ and $C_3$. Then
$$
M_2\leq M_1^{1-a}M_3^a,
$$
where
\begin{eqnarray*}
a & =& \log{\frac{r_2}{r_1}}\big/\log{\frac{r_3}{r_1}}\\
 & = & \log \left(1+\frac{1+(1/\log \log t)-\sigma}{\log \log t-1-(1/\log \log t)}\right) \big/\log\left(1+\frac{1-1/p}{\log \log t-1-(1/\log \log t)}\right)\\
 & = & \frac{1-\sigma}{1-1/p}+O(1/\log \log t),  \ \ \ \ t>\max\{t_1,t_2, t_3\}.
\end{eqnarray*}

By (\ref{M3}) we have $M_3<A (\log \log t) \log t$. On the other hand, 
$$
M_1\leq \max_{x\geq 1+(1/\log \log t)}\Big|\sum_{n=2}^{\infty}\frac{\Lambda_1(n) \chi(n)}{n^{x+it}}\Big|  \leq \sum_{n=2}^{\infty}\frac{1}{n^{1+(1/\log \log t)}}<A\log \log t, 
$$
for $t>\max\{t_1,t_2, t_3,t_4\}$. 

Hence, 
$$
|\log L(\sigma+it,\chi)|<(A\log \log t)^{1-a}(A\log \log t)^{a}(\log t)^{a})=A\log \log t(\log t)^{\frac{1-\sigma}{1-1/p}},
$$
in the region
$1/p+1/(\log \log t)\leq \sigma \leq 1$ and $t>\max\{t_1,t_2, t_3,t_4\}$. 

Given  $\delta, \epsilon >0$, there is $t_5$ with 
$$1/(\log \log t)<\delta  \ \ \ \ \mbox{and} \ \ \ \  A\log \log t(\log t)^{\frac{1-p\sigma}{p-1}}<\epsilon, \ \ \ \ t>t_5.$$
Therefore,
$$-\epsilon \log t\leq \log |L(s,\chi)|\leq \epsilon \log t,  \ \ \ \ t>\max\{t_1,t_2, t_3,t_4\}.$$
Finally, setting  $t_0:=\max\{t_1,t_2, t_3,t_4,t_5\}$, we conclude that
$$
|L(s,\chi)|^{\pm1}\leq t^\epsilon
$$
in the region $\sigma \in [1/p+\delta, 1]$, $t>t_0$. 

 It is not difficult to observe that the same reasoning yields that the estimate
$$
|L(s,\overline{\chi})|^{\pm1}\leq t^\epsilon
$$
holds in the same region, where $\overline{\chi}$ is the conjugate character of $\chi$. Since $|L(\overline{s},\chi)|=|L(s,\overline{\chi})|$,  then
$$
|L(s,\chi)|^{\pm1}\leq |t|^\epsilon
$$
for  $\sigma \in [1/p+\delta, 1]$, $t>t_0$.
Applying the Phragm\'en-Lindel\"of convexity principle we obtain the desired result.
 \end{proof}

Next we formulate and prove a generalisation of another theorem of Littlewood \cite{Litt} (see Theorem 14.25(A) in \cite{Tit}):

\begin{lem} 
\label{lem2}
Suppose that $L(s,\chi)$ does not vanish for $\Re(s)>1/p$. Then the series  
$$\sum\frac{\mu(k)\chi(k)}{k^s}$$
converges to $1/L(s,\chi)$, for $\Re(s)>1/p$.
\end{lem}
\begin{proof} 
 Let $\Re(s)>1/p$. Applying Lemma \ref{LemTit} for 
 $$
 f(s)=\frac{1}{L(s,\chi)}=\sum_{k=1}^\infty \mu(k)\chi(k)k^{-s},
 $$
 $c=2$ and $x$ the half of an odd number, we obtain
\begin{eqnarray*}
\sum_{k<x} \frac{\mu(k)\chi(k)}{k^s} & = & \frac{1}{2 \pi i} \int_{2-iT}^{2+iT}\frac{1}{L(s+w,\chi)}\frac{x^w}{w}dw+O\left(\frac{x^2}{T} \right)\\
& = & \frac{1}{2 \pi i} \left( \int_{2-iT}^{1/p-\sigma+\delta-iT}+ \int_{1/p-\sigma+\delta-iT}^{1/p-\sigma+\delta+iT}+ \int_{1/p-\sigma+\delta+iT}^{2+iT}\right)\frac{1}{L(s+w,\chi)}\frac{x^w}{w}dw\\
\ & & +\frac{1}{L(s,\chi)}+O\left(\frac{x^2}{T} \right),
\end{eqnarray*}
with $0<\delta<\sigma-1/p$. 
By Theorem \ref{Littl_Est} the first and the third integrals can be estimated by
$$O\left( T^{-1+\epsilon} \int_{1/p-\sigma+\delta}^{2}x^udu\right)=O\left( T^{-1+\epsilon} x^2\right)$$
and the second one by
$$O\left( x^{1/p-\sigma+\delta} \int_{-T}^{T}(1+|t|)^{-1+\epsilon} dt\right)=O\left( x^{1/p-\sigma+\delta}T^{\epsilon}\right).$$
Hence, 
$$\sum_{k<x} \frac{\mu(k)\chi(k)}{k^s}= \frac{1}{L(s,\chi)}+O\left( T^{-1+\epsilon} x^2\right)+O\left( x^{1/p-\sigma+\delta}T^{\epsilon}\right).$$
Choosing $T=x^3$, the $O$-terms tend to zero when $x\to \infty$ which completes the proof. 
\end{proof}

Lemma \ref{lem2} and Theorem \ref{Littl_Est} yield:

\begin{lem}
\label{lema00}
Let $\delta, \epsilon>0$, $p\in (1, 2]$, with $1/p+\delta \leq 1$, and  $L(s,\chi)$ be a Dirichlet $L$-series which does not vanish in 
the semi-plane $\sigma>1/p$. Then
$$
\sum_{k=1}^n \frac{\mu(k)\chi(k)}{k^s}=O((1+|t|)^\epsilon)
$$
uniformly with respect to both $n\in \mathbb{N}$ and the strip $\Re(s) \in [1/p+\delta, 1]$.
\end{lem}

Finally we recall the following result about sums involving Dirichlet characters (see \cite[Theorem 6.17]{Apo}):
\begin{LEM}
\label{sums}
Let $\chi$ be a non-principal Dirichlet character and $f\in C^1[1,\infty)$ be a non-negative function which decreases in $[1,\infty)$. Then
$$
\sum_{k\leq x} \chi(k) f(k) = O(f(1)).
$$  
\end{LEM}

\subsection{Lubinsky's  Dirichlet orthogonal polynomials}\label{LubPol}
Recently Lubinsky \cite{LubJAT} considered the general Dirichlet polynomials built on the basis 
$\lambda_k^{-it}$, where $1= \lambda_1 < \lambda_2 < \lambda_3 < \cdots$, and provided an ingenious construction of 
the corresponding orthogonal basis with respect to the arctangent density.
For any such a strictly increasing sequence of real numbers $\lambda_n$, Lubinsky proved that  the general Dirichlet polynomials 
$\phi_1(t)=1$, $\phi_n(t)=(\lambda_n^{1-it} - \lambda_{n-1}^{1-it})/\sqrt{\lambda_n^{2} - \lambda_{n-1}^{2}}$, $n\geq 2$, satisfy 
$$
\int_\mathbb{R} \phi_n(t) \overline{\phi_m(t)} \frac{dt}{\pi (1+t^2)} = \delta_{nm},\ \ n,m \in \mathbb{N}, 
$$ 
and described the asymptotic behaviour of the corresponding kernel polynomaials. 
 The choice $\lambda_n=n^{1/p}$ and a simple change of variables shows that the Dirichlet polynomials
 \begin{eqnarray*}
\psi_{1}(t) & = & 1,\\ 
\psi_{n}(t) & = & \frac{n^{1/p-it} - (n-1)^{1/p-it}}{\sqrt{n^{2/p} - (n-1)^{2/p}}},\ \ \ n\geq 2
\end{eqnarray*} 
satisfy
$$
\frac{1}{p\, \pi} \int_\mathbb{R} \psi_{n}(t) \overline{\psi_{m}(t)} \frac{dt}{1/p^2+t^2} = \delta_{nm},\ \ n,m \in \mathbb{N}.
$$ 
Let 
$$
K_n(u,v) = \sum_{k=1}^n \psi_k(u) \overline{\psi_k(v)}
$$ 
be the corresponding kernel polynomials, which, according to \cite[(1.20), (1.19)]{LubJAT}, obey the following uniform asymptotic estimates in compact subsets of $\mathbb{R}$, as $n \rightarrow \infty$:
\begin{equation}
\label{LubEst1}
K_n(u,u) = \frac{p}{2} \, |1/p+i u|^2\, \log n\, (1+o(1))
\end{equation}
and 
\begin{equation}
\label{LubEst2}
|K_n(u,v)| \leq p\, \frac{|1/p+i u|\, |1/p-i v|}{|u-v|} + o(\log n).
\end{equation}

We shall need the following simple fact: 
\begin{lem}
\label{InvMat}
For every $m\in \mathbb{N}$ and for any distinct numbers $t_1,\ldots, t_m \in \mathbb{R}$, there exists $n(m)\in \mathbb{N}$ such that the self-adjoint matrix 
$H = \left( K_n(t_i,t_j) \right)_{i,j=1}^{\ m}$ 
%$$
%H_{m}=\left( \begin{array}{ccc}
%K_n(\gamma_1,\gamma_1)  & \dots &K_n(\gamma_1,\gamma_m) )\\
%\vdots             &  & \vdots \\
%K_n(\gamma_m,\gamma_1)  & \dots & K_n(\gamma_m,\gamma_m) 
%\end{array}\right)
%$$
is nonsingular for every $n>n(m)$. Moreover,
\begin{equation}
\label{Red}
\frac{\det H}{(\log n)^m}= \frac{p^m}{2^m} |1/p+i t_1|^2\dots |1/p+i t_m|^2+o(1)\ \ \mathrm{as}\ \ n \rightarrow \infty.
\end{equation}
\end{lem}

\begin{proof} The Leibniz formula for the expansion of the determinant $H$ over the permutations $\mathcal{P}_m$ and the above asymptotic formulae for the kernel 
polynomials yield  
\begin{eqnarray*}
\det H & = & \sum_{\sigma \in \mathcal{P}_m}  \mathrm{sgn} (\sigma) K_n(t_1,t_{\sigma(1)})\dots K_n(t_m,t_{\sigma(m)})\\
& = & \frac{p^m}{2^m}\, |1/p+i t_1|^2\dots |1/p+i t_m|^2 (\log n)^m(1+o(1))+O((\log n)^{m-2})\\
& = & \frac{p^m}{2^m}\, |1/p+i t_1|^2\dots |1/p+i t_m|^2 (\log n)^m+o((\log n)^m),
\end{eqnarray*}
which is equivalent to (\ref{Red}). Thus, obviously $H$ is nonsingular for all sufficiently large $n$.  
\end{proof}

\section{Proofs}

Our first result is a generalisation of the Nyman-Beurling criterion for a relatively wide class of Dirichlet series. 
\begin{thm}
\label{GNB} 
If $\eta \in \mathcal{D}$ then for every $p>1$, the following statements are equivalent:\\
(i) $\eta(s)$ does not vanish in the semi-plane  $\Re(s)>1/p$;\\
(ii) $C_\eta^p=L^p (0, 1)$;\\
(iii) The characteristic function $\textbf{1}_{[0,1]}$ belongs to $C_\eta^p$.
\end{thm}

Similar results were proved recently by Delaunay, Fricain, Mosaki and Robert \cite{DFMR13, DFMR13II} and de Roton \cite{Rot07}. 
Though the class of Dirichlet series considered in  \cite{DFMR13, DFMR13II} is wider than we deal with, the proof that we furnish is rather simpler. One of the
result in \cite{Rot07} contains the statement of Theorem \ref{GNB} but only for the particular case $p=2$.
We provide a proof of Theorem \ref{GNB} because it turns out to a be a clue tool for the remaining results that we establish in this note. 

We begin with a result which is analogous to an observation of Beurling \cite{Beu55} concerning the case  $C_\zeta$:

\begin{lem}\label{lema01} For every $\gamma \in (0,1]$, let $T_\gamma:L_p(0,1)\mapsto L_p(0,1)$ be the  operator defined by 
$$
T_\gamma f(x)=
\left\{  
\begin{array}{lcl} 
f(x/\gamma) & \mbox{if} & 0<x\leq \gamma, \\ 0 
& \mbox{if} & \gamma<x<1, 
\end{array} 
\right.
$$
and $T_0(f)$ is the identically zero function. Then $\| T_\gamma(f)\|_p \leq \|f\|_p$ and $T_\gamma(C_\eta)\subset C_\eta$.
\end{lem}
and summarise the results in Beurling's paper \cite{Beu53} as follows: 
\begin{LEM}
\label{lema02} 
Let $g\in L_q(0,1)$, with $1<q<\infty$,  be such that 
$$\int_0^a|g(x)|dx>0,  \ \  \mathrm{for\ every}\  a\in(0,1).
$$
For $1\leq r< q$,  consider $E_g^r$, the closure in $L_r(0,1)$ of the linear space generated by  $\{g(\gamma x), \gamma \in (0,1]\}$. Then there exists $\lambda\in \mathbb{C}$ with  $\Re(\lambda)>-1/q$, such that  
$$
x^\lambda \in  \bigcap_{1\leq r<q} E_g^r.
$$
\end{LEM}

\vspace{.2in}

 \textit{Proof of Theorem \ref{GNB}.} First we prove that $(ii)$ and $(iii)$ are equivalent. It is clear that $(ii)$ implies $(iii)$. Suppose that $\textbf{1}_{[0,1]}\in C_\eta^p$.  For every $g\in L_p(0,1)$ and any $\epsilon>0$ there exists 
 a partition $0=\gamma_0<\gamma_1<\dots<\gamma_n=1$ of the do interval $(0,1)$ and constants $a_1,\dots,a_n$, not all equal to zero, such that 
$$
\|g-\sum_{k=1}^{n}a_k \textbf{1}_{(\gamma_{k-1},\gamma_{k})}\|_p<\frac{\epsilon}{2}.
$$
Since $\textbf{1}_{(0,1)}\in C_\eta^p$ there is $h\in C_\eta$ with 
$$
\|\textbf{1}_{(0,1)}-h\|_p<\frac{\epsilon}{4n\max\{|a_k|\}}.
$$
Let us choose  $f=\sum_{k=1}^n a_k(T_{\gamma_k}-T_{\gamma_{k-1}})h$. Lemma \ref{lema01} implies that $f\in C_\eta$ and 
\begin{eqnarray*}
\|g-f\|_p & = & \|g-\sum_{k=1}^n a_k(T_{\gamma_k}-T_{\gamma_{k-1}})h\|_p\\
 & = & \|g-\sum_{k=1}^{n} a_k \textbf{1}_{(\gamma_{k-1},\gamma_{k})} +\sum_{k=1}^{n} a_k \textbf{1}_{(\gamma_{k-1},\gamma_{k})} - \sum_{k=1}^n a_k (T_{\gamma_k}-T_{\gamma_{k-1}}) h\|_p\\
& = & \|g-\sum_{k=1}^{n}a_k \textbf{1}_{(\gamma_{k-1},\gamma_{k})} +\sum_{k=1}^{n}  a_k (T_{\gamma_k}-T_{\gamma_{k-1}}) (\textbf{1}_{(0,1)} -h) \|_p\\
& \leq & \|g-\sum_{k=1}^{n}a_k \textbf{1}_{(\gamma_{k-1},\gamma_{k})} \|_p+\sum_{k=1}^{n}  |a_k|\, \| (T_{\gamma_k}-T_{\gamma_{k-1}}) (\textbf{1}_{(0,1)} -h)\|_p\\
& < & \epsilon.
\end{eqnarray*}
Next we prove that (iii) implies (i).  Let 
$$
h(x) = \sum_{k=1}^n b_k\, \kappa_\eta (1/\theta_k x) \in C_\eta.
$$
By (\ref{Int1f})
\begin{equation}
\label{Int2f}
\int_0^1(1+h(t))t^{s-1}dt=\frac{1-\eta(s)\sum_{k=1}^{n}b_k\theta_k^{-s}}{s}, \ \  \Re(s)>0.
\end{equation}
Suppose that $\textbf{1}_{(0,1)}\in C_\eta^p$. Then, given $\epsilon>0$, there exists $h \in C_\eta$ such that $\|\textbf{1}_{(0,1)} + h \|_p<\epsilon$. It is clear that  $x^{s-1}\in L_q(0,1)$ provided $1/p+1/q=1$ and $\Re(s)>1/p$. 
Furthermore,
$$
\| x^{s-1} \|_q^q = \frac{1}{q(\Re(s)-1/p)}.
$$
Applying H\"older's inequality to (\ref{Int2f}) we obtain 
$$
|1-\eta(s)\sum_{k=1}^{n}b_k\theta_k^{-s}|^q < \epsilon^q \frac{|s|^q}{q(\Re(s)-1/p)}.
$$
Let us assume that $\eta$ possesses a zero $\rho$ in the semi-plane $\Re(s)>1/p$. Letting $\epsilon \to 0$ in the latter inequality we obtain an obvious contradiction. Therefore $\eta(s)$ does not vanish in  $\Re(s)>1/p$.

Finally we prove that (i) implies (ii). If $C_\eta$ is not dense in $L_p (0, 1)$, that is $C_\eta^p\neq L_p (0, 1)$, theh by the Riesz representation theorem, there is $g\in L_q(0,1)$, such that $g$ is a nonzero element of $L_q(0,1)$ and
$$
\int_0^1g(x)h(x)dx=0 \ \ \mathrm{for\ all}\  h \in C_\eta.
$$
With the aid of the operator $T_\gamma$, introduced in Lemma \ref{lema01}, we conclude that
\begin{equation}
\label{auxiliar}
\int_0^1g(x)\, T_\gamma h (x)dx=\gamma \int_0^1g(\gamma x) h(x)dx=0, \ \ \mathrm{for\ every}\  h\in C_\eta.
\end{equation}
In order to apply Lemma \ref{lema02} to the function $g$ we need to prove that 
$$
\int_0^a|g(x)|dx>0,  \ \ \ \mathrm{for\ every}\  a\in(0,1),
$$
or equivalently, that $g$ is not zero almost everywhere in $(0,a)$, for every $a\in(0,1)$. Suppose the contrary, that $g\equiv0$ a.e. in $(0,a)$ for some $a\in (0,1)$. 
Choose $b$, such that $a<b<\min (1,2a)$ and set 
$$
h(x)=b^\alpha\, \kappa_\eta \left(\frac{a}{x}\right)-a^\alpha\, \kappa_\eta\left(\frac{b}{x}\right).
$$ 
It is obvious that $h\in C_\eta$ and it vanishes for $x>b$. Let $a_1$ be the first coefficient in the representation of $\eta$ which defines of the latter $h$.
Recall that, by definition $a_1\neq 0$. Moreover, $h$ takes the value $a_1 a^\alpha$ in $x\in(a,b)$. 
Therefore,
$$
0= \int_0^1g(x) h(x)dx=a_1a^A\int_a^b g(x) dx, \ \ \ a<b<\min (1,2a),
$$
which implies that $g= 0$ almost everywhere in $(0,\min \{1,2a\})$. Substituting $a$ by $2a,4a, \dots$ we conclude that $g=0$ almost everywhere in $(0,1)$. This is a contradiction with the fact that $g$ is a nonzero element of $L_p(0,1)$. 
Thus,
$$
\int_0^a|g(x)|dx>0,  \ \  \mathrm{for\ every}\  a\in(0,1).
$$
For $1\leq r< q$,  consider $E_g^r$, the closure in $L_r(0,1)$ of the linear space generated by  $\{g(\gamma x),\  \gamma \in (0,1]\}$.
By Lemma \ref{lema02} there exists a function $x^\lambda$, with $\Re(\lambda)>-1/q$ and $x^\lambda\in  \cap_{1\leq r<q} E_g^r$. We shall prove that
$$
\int_0^1x^\lambda h(x)dx=0 \ \  \mathrm{for\ all}\  h\in C_\eta.
$$
For each $h\in C_\eta$ there is $M>0$ such that $|h(x)|<M$ when $x\in (0,1)$. Since  $x^\lambda \in E_g^1$, then for every $\epsilon>0$ there exist exist $\gamma_1,\dots,\gamma_n \in (0,1]$ with $\|x^\lambda-\sum_{k=1}^ng(\gamma_kx)\|_1<\epsilon/M$. 
The latter observation and (\ref{auxiliar}) yield
\begin{eqnarray*}
|\int_0^1x^\lambda h(x)dx| & = & |\int_0^1\left(x^\lambda-\sum_{k=1}^ng(\gamma_kx)+\sum_{k=1}^ng(\gamma_kx) \right)h(x)dx|\\
& = & |\int_0^1\left(x^\lambda-\sum_{k=1}^ng(\gamma_kx) \right)h(x)dx|\\
& \leq & L \int_0^1\left|x^\lambda-\sum_{k=1}^ng(\gamma_kx) \right| dx\\
&< & \epsilon.
\end{eqnarray*}
Hence
$$
\int_0^1x^\lambda h(x)dx=0, \ \  \mathrm{for\ all}\  h\in C_\eta.
$$
In particular, choosing $h(x)=\kappa_\eta(1/x)-\theta^\alpha \kappa_\eta(1/\theta x)$, with $\theta\in [1,\infty)$, in such a way that $\theta^{\alpha-\lambda -1}-1\neq 0$,  in view of (\ref{Int1f})  we obtain
$$
0=\int_0^1x^\lambda \left(\kappa_\eta\left(\frac{1}{x}\right)-\theta^\alpha \kappa_\eta\left(\frac{1}{\theta x}\right)\right)dx=\frac{\theta^{\alpha-\lambda -1}-1}{\lambda+1}\eta(\lambda+1).
$$
Therefore $\eta(\lambda+1)=0$ for $\Re(1+\lambda)>1-1/q=1/p$, that is, $\eta$ possesses a zero $s$ with  $\Re(s)>1/p$. This contradicts (i).
\ \ \ \ \  $\Box$

\vspace{.2in}

 \textit{Proof of Theorem \ref{GBNew}.}   Let $p\in(1,2]$ and $L(s,\chi)$ be a Dirichlet $L$-series with character modulo $q$. We define the functions
 $$
\eta(s)=L(s+1/p-1/2,\chi)=\sum_{k=1}^\infty \frac{\chi(k)}{k^{1/p-1/2}}\frac{1}{k^s}
$$ 
and
$$
\kappa_\eta(x)=\beta \, x^\alpha-\sum_{k\leq x}\chi(k)k^{1/2-1/p},
$$
with  $\alpha=3/2-1/p$ and $\beta= \varphi(q)/(\alpha q)$ if  $\chi$ is principal and with 
$ \alpha=\beta=0$ if  $\chi$ is a non-principal character. We claim that $\eta \in \mathcal{D}$.  For the principal character we use the Abel identity for the function $f(x)=x^{-1/p+1/2}$ to obtain
\begin{eqnarray}
\sum_{k\leq x} \frac{\chi(k)}{k^{1/p-1/2}} &=&-\left(\frac{\varphi(q)}{q}x-\sum_{k\leq x}\chi(k)\right)\frac{1}{x^{1/p-1/2}}+\frac{\varphi(q)}{q}\nonumber\\
& &-(1/p-1/2)\int_1^x \left(\frac{\varphi(q)}{q}t-\sum_{k\leq t}\chi(k)\right)\frac{1}{t^{1/p+1/2}}dt \label{Abel_id}\\
& &+ \frac{\varphi(q)}{q}\int_1^x \frac{1}{t^{1/p-1/2}}dt.\nonumber
\end{eqnarray}
Observe that for the principal character modulo $q$ the function 
$$\frac{\varphi(q)}{q}x-\sum_{k\leq x}\chi(k)$$
is bounded because it is periodic with period $q$. Therefore
$$
\sum_{k\leq x} \frac{\chi(k)}{k^{1/p-1/2}}=\frac{\varphi(q)}{q(3/2-1/p)}x^{3/2-1/p}+O(1)
$$
because the right-hand side of (\ref{Abel_id}) is dominated by its last term while the remaining ones are bounded. Hence 
$$
\kappa_\eta(x) = \frac{\varphi(q)}{q(3/2-1/p)} x^{3/2-1/p} - \sum_{k\leq x} \frac{\chi(k)}{k^{1/p-1/2}}
$$
isa bounded. Thus $\eta \in \mathcal{D}$ with $\alpha=3/2-1/p$ and $\beta= \varphi(q)/(\alpha q)$. For a non-principal character  we can use Lemma \ref{sums} with $f(x)=x^{-1/p+1/2}$ obtaining  
$$\sum_{k\leq x} \frac{\chi(k)}{k^{1/p-1/2}}=O(1)$$
which shows that $\eta \in \mathcal{D}$ with $\alpha=\beta=0$.

In order to prove the first statement of the theorem we need  to show that $\mathbf{1}_{(0,1)}$ belongs to the closure in $L^2(0,\infty)$ of the set $\mathrm{span} \{\kappa_\eta(1/kx): k \in \N \}$. First we show that $\mathbf{1}_{(0,1)} \in C_{\eta,\N}^2$. 
The proof of this part of the theorem depends on the type of the character.  Suppose first that $\chi$ is principal. Our proof is inspired by the ingenious idea developed by Bagchi \cite{Bag} for his proof of  the B\'aez-Duarte criterion for the Riemann zeta function.

If $L(s,\chi)$ does not vanish for $\Re(s)>1/p$ then $\eta(s)$ does not vanish for $\Re(s)>1/2$. But $\eta \in \mathcal{D}$ and by Theorem \ref{GNB} the characteristic function $\textbf{1}_{(0,1)}$ belongs to $C_\eta^2$. Having in mind that that $\beta\neq 0$, 
it is not difficult to observe that  $C_{\eta,\N}=\mbox{span} \{\kappa_\eta(1/kx)-(1/k^\alpha)\kappa_\eta(1/x):k \in \N\}$. For any $\theta\geq 1$, let us consider the Mellin  transform of $\kappa_\eta(1/\theta x)-\theta^{-\alpha}\kappa_\eta(1/x)$. By (\ref{Mellin}),
$$
\mathcal M\left[\kappa_\eta(1/\theta x)-\theta^{-\alpha}\kappa_\eta(1/x); s\right]=-\frac{\eta(s)}{s}(\theta^{-s}-\theta^{-\alpha}).
$$
Recall again that, by Plancherel's theorem the Mellin transform $\mathcal{M}$ can be extended to an isometry $(1/\sqrt{2 \pi})\mathcal{M}: L^2 (0,\infty) \to L^2(\Re(s)=1/2)$. Therefore, in order to prove that $\textbf{1}_{(0,1)}\in C_{\eta,\N}^2$ it suffices to establish the claim that $\mathcal M[\mathbf{1}_{(0,1)};s]=1/s$ belongs to the closure of the 
set $\mbox{span} \{ -\frac{\eta(s)}{s}(k^{-s}-k^{-\alpha}) :k \in \N\}$ in the space $L^2(\Re(s)=1/2)$. 

For each $n\in \N$ and $\epsilon\in (0,1-1/p)$ we define the function 
$H_{n,\epsilon} \in L^2(\Re(s)=1/2)$ by
$$
H_{n,\epsilon}=\sum_{k=1}^{n}\frac{\mu_{\eta}(k)}{k^\epsilon}G_k,
$$
where 
$\mu_{\eta}(k)$ are the coefficients of the expansion of $1/\eta$ in a formal Dirichlet series and 
$$
G_k(s)=(k^{-s}-k^{-\alpha}) \frac{\eta(s)}{s}.
$$ 
Observe that $H_{n,\epsilon} \in \mbox{span}\{G_k:k\in \N\}$ and 
$$
H_{n,\epsilon}(s)=\frac{\eta(s)}{s}\left( \sum_{k=1}^{n}\frac{\mu_\eta(k)}{k^{s+\epsilon}}- \sum_{k=1}^{n}\frac{\mu_\eta(k)}{k^{\alpha+\epsilon}}\right),\ \ \ \ s\in \Re(s)=1/2.
$$
The Dirichlet $L$-series $L(s,\chi)$ does not vanish when $\Re(s)>1/p$ (hypothesis). 
Hence, by Lemma \ref{lem2}, 
$$
\lim_{n\to\infty} H_{n,\epsilon}(s)=H_\epsilon(s),   \ \ \ \ \Re(s)=1/2,
$$
with
\begin{eqnarray*}
H_\epsilon(s) & = & \frac{\eta(s)}{s}\left(\frac{1}{\eta(s+\epsilon)}-\frac{1}{\eta(\alpha+\epsilon)} \right)\\
\ & = &\frac{L(s+1/p-1/2,\chi)}{s}\left(\frac{1}{L(s+1/p-1/2+\epsilon,\chi)}-\frac{1}{L(1+\epsilon,\chi)} \right).
\end{eqnarray*}
It follows from Theorem \ref{Littl_Est} and Lemma \ref{lema00} about estimates of Dirichlet $L$-series  that the modulus of $H_{n,\epsilon}$ is bounded by a function from $L^2(\Re(s)=1/2)$. 
Hence, by Lebesgue's dominated convergence theorem, for every fixed $\epsilon>0$, 
$$\lim_{n\to \infty} H_{n,\epsilon}=H_\epsilon,$$
in the $L^2(\Re(s)=1/2)$ norm. Since $H_{n,\epsilon} \in \mbox{span}\{G_k:k\in \N\}$ for every fixed $\epsilon>0$, then $H_{\epsilon}$ belongs to the closure of  $\mbox{span}\{G_k:k\in \N\}$. 
The function $L(s,\chi)$ has a pole in $s=1$, so that, by the definition of $H_\epsilon(s) $, 
$$
\lim_{\epsilon \to 0} H_\epsilon(s)=\frac{1}{s}=E(s), \ \ \ \ s\in \Re(s)=1/2.
$$
Thus, in order to prove that $E$ belongs to the closure of  $\mbox{span}\{G_k:k\in \N\}$, it suffices to show that the modulus of $H_\epsilon$, $0<\epsilon<1-1/p$, is uniformly bounded along the critical  line by 
a function from $L^2(\Re(s)=1/2)$. Lebesgue's dominated convergence theorem yields that
$$\lim_{\epsilon \to 0}H_\epsilon=E$$
in the norm of $L^2(\Re(s)=1/2)$.  
The Hadamard factorisation of $\xi_{\chi}(s)$ in Theorem \ref{ThIw} implies that there is a positive constant $c_0$ such that 
$$
\left|   \frac{\xi_{\chi}(s+1/p-1/2)}{\xi_{\chi}(s+1/p-1/2+\epsilon)}\right|<c_0,  \ \ \Re(s)=1/2,  \ \ \ \epsilon \in (0,1-1/p).
$$
The definition of $\xi_{\chi}$ in (\ref{xiL}), Stirling's formula about the asymptotic behaviour of the Gamma function and the restriction on $\epsilon$  yield 
\begin{eqnarray*}
\left|\frac{L(s+1/p-1/2,\chi)}{L(s+1/p-1/2+\epsilon,\chi)}\right| & < & c_0\, (q\pi)^{\epsilon/2}\frac{(s+1/p-1/2+\epsilon)(s+1/p-1/2+\epsilon-1)}{s+1/p-1/2(s+1/p-1/2-1)}\\
& & \ \times  \frac{\Gamma((s+1/p-1/2+\epsilon+\alpha(\chi))/2)}{\Gamma((s+1/p-1/2+\alpha(\chi))/2)}\\
& < & c_1\, t^{1/4}.
\end{eqnarray*}
This estimate, the ones in Theorem \ref{ThIw}  and Theorem \ref{Littl_Est}, yield the desired estimate by an $L^2(\Re(s)=1/2)$ function:
$$
| H_\epsilon(s) | = O ( t^{-3/4}),\ \ \Re(s)=1/2.
$$

The proof for the case of non-principal  character goes along similar reasonings with the only difference that in this situation we choice $G_k$ to be
$$
G_k(s)=k^{-s} \frac{\eta(s)}{s}.
$$ 

    We have already established that $\mathbf{1}_{(0,1)} \in C_{\eta,\N}^2$ which implies 
$$
\lim_{n\to \infty} \inf_{{b_1,\ldots,b_n\in \C}\atop}\int_0^1 \left|\textbf{1}_{(0,1)}-\sum_{k=1}^n b_k \kappa_\eta \left(\frac{1}{kx}\right) \right|^2 dx=0,
$$
where $b_1, \ldots, b_n$ obey the additional restriction $\beta\sum_{k=1}^n b_k k^{-\alpha}=0$. This restriction implies   
$$
\lim_{n\to \infty} \inf_{{b_1,\ldots,b_n\in \C}\atop}\int_0^\infty \left|\textbf{1}_{(0,1)}-\sum_{k=1}^n b_k \kappa_\eta \left(\frac{1}{kx}\right) \right|^2 dx=0,
$$
which obviously yields
$$
\lim_{n\to \infty} d_n(L,p)=0.
$$

The proof of the other statement of the theorem needs to be done for the principal character and for non-principal characters separately too. 
For the principal character, suppose that $\lim_{n\to \infty} d_n(L,p)=0$. This means that for any given $\epsilon>0$ there exist complex numbers $b_1,\ldots, b_n$, such that  $d_n^2(\eta,2)<\epsilon$, or equivalently,  there exists $\tilde{h}(x)=\sum_{k=1}^n b_k \kappa_\eta(1/(kx))$, such that 
$\| \mathbf{1}_{(0,1)} - \tilde{h} \|_{L^2[0,\infty)} < \epsilon$. By the definition of $\kappa_\eta$, it is clear that 
\begin{equation}
\label{1-f}
\| \mathbf{1}_{(0,1)} - \tilde{h} \|_{L^2(0,\infty)} \leq \beta \left|\sum_{k=1}^n \frac{b_k}{k^\alpha} \right| \left(\int_1^\infty x^{-2\alpha} dx \right)^{1/2}
\end{equation} 
and  $x^{-2\alpha}$ is integrable because $2\alpha=3-2/p>1$. 
Consider the function 
$$
h(x) = \tilde{h}(x) -  \left( \sum_{k=1}^n \frac{b_k}{k^\alpha} \right) \kappa_\eta\left(\frac{1}{x}\right).
$$
Obviously $h \in C_\eta$, that is, it satisfies the lats restriction in (\ref{Ceta}). Moreover, 
\begin{eqnarray*}
\| h - \tilde{h} \|_{L^2(0,\infty)} & = & \left|\sum_{k=1}^n \frac{b_k}{k^\alpha} \right| \left(\int_0^\infty \left| \kappa_\eta\left(\frac{1}{x}\right) \right|^2 dx \right)^{1/2}\\
& = & \left|\sum_{k=1}^n \frac{b_k}{k^\alpha} \right| \left( \int_0^1 \left| \kappa_\eta\left(\frac{1}{x}\right) \right|^2 dx + \beta \int_1^\infty x^{-2\alpha} dx     \right)^{1/2},
\end{eqnarray*} 
where the first integral is well defined because $\kappa_\eta(1/x)$ is bounded in $(0,1)$. It follows then from (\ref{1-f})  that, with a suitable constant $K$,  
$$
\| h - \tilde{h} \|_{L^2(0,\infty)} \leq K\, \| \mathbf{1}_{(0,1)} - \tilde{h} \|_{L^2(0,\infty)} \leq K \epsilon.
$$
Then
$$
\| \mathbf{1}_{(0,1)} - h \|_{L^2(0,\infty)} \leq  \| \mathbf{1}_{(0,1)} - \tilde{h} \|_{L^2(0,\infty)} + \| h - \tilde{h} \|_{L^2(0,\infty)}  \leq  (1+K)\, \epsilon.
$$
Therefore, $\mathbf{1}_{(0,1)}\in C_\eta^2$  and by Theorem \ref{GNB} the function $\eta(s)$ does not vanish for $\Re(s)>1/2$ or equivalently $L(s,\chi)$ have no zeros in $\Re(s)>1/p$. 

The proof in the case of a non-principal character is immediate. \ \ \ \ \ \ \ \ \ \ \ $\Box$

\vspace{.2in}

{\em Proof of Theorem} \ref{GBLp}. First we observe that the previous theorem and an application of Mellin's transform 
implies that for every $L(s,\chi)$ the corresponding function $\eta(s)=L(s+1/p-1/2,\chi)$ is free of zeros in  $\Re(s)>1/2$ if and only if 
$\lim d_n(\eta,2)\rightarrow 0$ as $n\rightarrow \infty$, where 
$$
d_n^2(\eta,2)=\inf_{{A_n\in \mathcal{D}_n}\atop} \frac{1}{2\pi} \int_{\Re(s)=1/2}\left| \frac{1-\eta(s)A_n(s)}{s} \right|^2 |ds|,
$$
and the infimum is taken over the Dirichlet polynomials $A_n$ of degree $n$, $A_n(s) = \sum_{k=1}^{n} b_k k^{-s}$. But
\begin{eqnarray*}
d_n^2(\eta,2) & = & \inf_{ {A}_n \in \mathcal{D}_n\atop} \frac{1}{2\pi} \int_{-\infty}^\infty \left| \frac{1-L(1/p+it,\chi)A_n(1/2+it)}{1/2+it} \right|^2 dt\\
\ & = & \inf_{A_n \in \mathcal{D}_n\atop } \frac{1}{2\pi} \int_{-\infty}^\infty   \left| \frac{1/p+it}{1/2+it} \right|^2    \left| \frac{1-L(1/p+it,\chi)A_n(1/p+it)}{1/p+it} \right|^2 dt.
\end{eqnarray*}
Since obviously $1<|(1/p+it)/(1/2+it)|<2$, then 
$$
d_n(L,p)<d_n(\eta,2)<2d_n(L,p)$$
and the result follows.\ \ \ \ \ $\Box$

\vspace{.2in}

{\em Proof of Theorem \ref{DnL}.}  Let $m\in \N$, $t_1, \ldots, t_m$ be distinct real numbers and $n\in \N$. Let us consider the Lubinsky Dirichlet orthogonal polynomials $(\psi_k)_{k=1}^n$ defined in Section \ref{LubPol}. Observe that $\mathcal{D}_{n,m}^p$ is nonempty 
since it contains $B_{n,p}(s)\equiv1$.  If $B_{n,p} \in \mathcal{D}_{n,m}^p$ then $B_n(1/p+it_j)=1$, $j=1,\dots,m$ and $B_n(1/p+it)=\sum_{k=1}^{n}b_k\psi_k(t)$. The  interpolations conditions can be rewritten in the form  
$$
A_{mn}\textbf{B}=\mathbf{1}_m, 
$$
where
$$
A = A_{mn}=\left(\begin{array}{cccc}

\psi_1(t_1) & \psi_2(t_1) & \dots & \psi_n(t_1)\\

\psi_1(t_2) & \psi_2(t_2) & \dots & \psi_n(t_2)\\

\vdots                    &  \vdots                   & & \vdots \\

\psi_1(t_m) & \psi_2(t_m) & \dots & \psi_n(t_m)

\end{array}\right),
$$
$\textbf{B}=(b_1,\ldots, b_{n})^T$ and $\mathbf{1}_m$ is the column vector of size $m$ all of whose entries are equal to one.
Then obviously 
$$
\frac{1}{p\, \pi} \|B_n(1/p+it) \|_{L^2(\R,\omega)}^2= |b_1|^2+|b_2|^2+\dots+|b_n|^2,  
$$
where $L^2(\R,\omega)$ is the weighted $L^2$ space with weight $\omega(t)= 1/(1/p^2+t^2)$. Thus the problem reduces to minimize $\|\textbf{B}\|^2$, $\textbf{B}\in \mathbb{C}^n$, subject to
$A_{mn}\textbf{B}=\mathbf{1}_m$. It is well known that the solution of the latter problem is given via a projection (see \cite[Theorem 2.19]{Bre} and \cite{GreRos, Nev}). In our setting, it is equivalent to solve the system 
\begin{eqnarray*}
A\,\textbf{B} & = & \mathbf{1}_m,\\
\textbf{B} & = & A^\ast \lambda,\ \ \lambda \in \mathbb{C}^m.
\end{eqnarray*}
By Lemma \ref{InvMat} there exists $n(m)\in \mathbb{N}$ such that the self-adjoint matrix $A A^*=H = \left( K_n(t_i,t_j) \right)_{i,j=1}^{\ m}$ is nonsingular for every $n>n(m)$. Hence, for $n>n(m)$, the system of matricial equations has an unique solution  $\tilde{\textbf{B}} = A^\ast H^{-1} \mathbf{1}_m$. Therefore, since  $A\, A^\ast = H$ and $H$ is self-adjoint, then
$$
\frac{1}{p\, \pi}  \|\tilde{B}_n(1/p+it) \|_{L^2(\R,\omega)}^2=   \tilde{\textbf{B}}^\ast \tilde{\textbf{B}} =  \mathbf{1}_m^\ast (H^{-1})^\ast  A\, A^\ast H^{-1} \mathbf{1}_m = \mathbf{1}_m^\ast H^{-1} \mathbf{1}_m. 
$$
Thus, by the Cramer formula for the inverse matrix
$$
\frac{1}{p\, \pi}  \|\tilde{B}_n(1/p+it) \|_{L^2(\R,\omega)}^2 = \sum_{i,j=1}^m (-1)^{i+j} \frac{\det H_{ij}}{\det H},
$$
where $H_{ij}$ are the $(i,j)$-th cofactors of $H$. 
On the other hand, the asymptotic relations (\ref{LubEst1}) and (\ref{LubEst2}) and Lemma  \ref{InvMat}  yield the following ones, as $n \rightarrow \infty$:
\begin{eqnarray*}
\det H_{jj} & \sim & (\log n)^{m-1}\, \frac{p^{m-1}}{2^{m-1}} \frac{|1/p+i t_1|^2\dots |1/p+i t_m|^2}{|1/p+i t_j|^2}\\
\det H_{ij} & = & O((\log n)^{m-2}),\ \ i\neq j\\
\det H & \sim &  (\log n)^m \frac{p^m}{2^m} |1/p+i t_1|^2\dots |1/p+i t_m|^2.
\end{eqnarray*}
Hence, 
$$
\frac{1}{p\, \pi}  \|\tilde{B}_n(1/p+it) \|_{L^2(\R,\omega)}^2 \sim \frac{2}{p \log n} \sum_{j=1}^m  \frac{1}{1/p^2+t_j^2},\ \ \mathrm{as}\ \ n\rightarrow \infty,
$$
which is equivalent to 
$$
d^2_{n,m,p}\sim \frac{1}{\log n} \sum_{j=1}^m  \frac{1}{1/p^2+t_j^2},\ \ \mathrm{as}\ \ n\rightarrow \infty.
$$
$\Box$

\end{document}